\documentclass[a4paper,11pt,fullpage]{article}
\usepackage[english]{babel}

\usepackage[left=1in,right=1in,top=1in,bottom=1in]{geometry}

\usepackage{amsmath}
\usepackage{amsthm}
\usepackage{amsfonts}
\usepackage{amssymb}
\usepackage{amsxtra}
\usepackage{latexsym}
\usepackage{ifthen}
\usepackage{cite}
\usepackage{esint}
\usepackage[cp1250]{inputenc}

\usepackage{epic}
\usepackage{eepic}

\DeclareMathOperator*{\osc}{osc}

\numberwithin{equation}{section}
\newtheorem{theorem}{Theorem}[section]
\newtheorem{lemma}{Lemma}[section]
\newtheorem{remark}{Remark}[section]

\def\XXint#1#2#3{{\setbox0=\hbox{$#1{#2#3}{\int}$}
     \vcenter{\hbox{$#2#3$}}\kern-.5\wd0}}

\begin{document}

\title{A note on the weak Harnack inequality for unbounded  minimizers of elliptic functionals with generalized Orlicz growth
}

\author{Mariia O. Savchenko, Igor I. Skrypnik, Yevgeniia A. Yevgenieva
 }

  \maketitle

  \begin{abstract} We prove the weak Harnack inequality for the functions $u$ which belong to the corresponding De Giorgi classes $DG^{-}(\Omega)$ under the additional assumption that $u\in L^{s}_{loc}(\Omega)$ with some  $s> 0$.  In particular, our result covers new cases of  functionals with a variable exponent or double-phase functionals under the non-logarithmic condition.



%\textbf{Declarations of interest}: none.

\end{abstract}

\pagestyle{myheadings} \thispagestyle{plain}
\markboth{Mariia O. Savchenko, Igor I. Skrypnik, Yevgeniia A.Yevgenieva}
{A note on the weak Harnack inequality....}

\section{Introduction and Main Results}\label{Introduction}
It is known that for integrands with $p,q$-growth, it is crucial that the gap between $p$ and $q$ is not too large. Otherwise, in the case
$q> \dfrac{np}{n-p}$, $p<n$ there exist unbounded minimizers and at the same time, the constant in Harnack inequality cannot be independent of the function, in contrast to the standard case, i.e. if $p=q$ (we refer the reader to \cite{Alk, AlkSur, BarColMin1, BarColMin2, BarColMin3, BenHarHasKar, BurSkr, ColMin1, ColMin2, ColMin3, CupMarMas, HadSkrVoi, HarKinLuk, HarHasLee, HarHasToi, Lie, LisSkr, MizOhnShi, Ok, RagTac, SavSkrYev, SkrVoi1, SkrVoi2, Sur} for results, references, historical notes and extensive survey of regularity issues). It was Ok \cite{Ok}, who proved the boundedness of minimizers of elliptic functionals of double-phase type in the case $q>\dfrac{np}{n-p}$ under some additional assumption. More precisely, under the condition $\osc\limits_{B_{r}(x_{0})}a(x)\leqslant A r^{a}$, $a\geqslant q-p$ the function $u$ belonging to the corresponding De Giorgi class $DG^{+}(\Omega)$ is bounded by a constant depending on $||u||_{L^{s}}$ with $s \geqslant \dfrac{(q-p) n}{a+p-q}$. This condition, for example, gives a possibility to improve the regularity results \cite{BarColMin1, BarColMin2, BarColMin3, ColMin1, ColMin2} for unbounded minimizers with constant depending on $||u||_{L^{s}}$. The weak Harnack inequality for unbounded supersolutions of the corresponding elliptic equations with generalized Orlicz growth  under the so-called logarithmic conditions was proved by the Moser method  in \cite{BenHarHasKar}.

It seems that for the corresponding De Giorgi classes $DG^{-}(\Omega)$ this question remains open even under the so-called logarithmic conditions, i.e. if $\lambda(r) \equiv 1$ (see condition ($\varPhi_{\lambda}$) below). In this note, we will prove the weak Harnack inequality for functions belonging to the corresponding  elliptic De Giorgi classes $DG^{-}(\Omega).$

We write  $W^{1,\varPhi(\cdot)}(\Omega)$ for the class of functions $u\in W^{1,1}(\Omega)$ with
$\int\limits_{\Omega}\varPhi(x, |\nabla u|)dx < \infty$ and we say that a measurable function $u:\Omega\rightarrow \mathbb{R}$ belongs to the elliptic class $DG^{\pm}_{\varPhi}(\Omega)$ if $u\in W^{1,\varPhi(\cdot)}(\Omega)$ and there exist numbers  $c >0$, $q>1$ such that for any ball $B_{8r}(x_{0})\subset\Omega$, any $k \in \mathbb{R}$  and any $ \sigma\in(0,1)$  the following inequalities hold:

\begin{equation}\label{eq1.1}
\int\limits_{A^{\pm}_{k,r(1-\sigma)}}  \varPhi\big(x,|\nabla u|\big) dx \leqslant  \frac{c}{\sigma^{q}}\,\int\limits_{A^{\pm}_{k,r}}  \varPhi\bigg(x,\frac{(u-k)_{\pm}}{r}\bigg) dx,
\end{equation}
here  $(u-k)_{\pm}:=\max\{\pm(u-k), 0\}$, $A^{\pm}_{k,r}:=B_{r}(x_{0})\cap \{(u-k)_{\pm}>0\}$.

Further, we suppose that $\varPhi(x, v):\Omega\times \mathbb{R}_{+}\rightarrow \mathbb{R}_{+}$ is a non-negative function satisfying the following properties: for any $x\in\Omega$
 the function $ v\rightarrow \varPhi(x, v)$ is increasing and
 $\lim\limits_{ v\rightarrow0}\varPhi(x, v)=0$,
 $\lim\limits_{ v\rightarrow +\infty}\varPhi(x, v)=+\infty$. We also assume that

\begin{itemize}
\item[($\varPhi$)]
There exist $1<p<q$ such that
for $x\in \Omega$ and for $ w\geqslant v > 0$ there holds
\begin{equation*}
\left( \frac{w}{v} \right)^{p} \leqslant\frac{\varPhi(x, w)}{\varPhi(x, v)}\leqslant
 \left( \frac{w}{v} \right)^{q}.
\end{equation*}
\end{itemize}

\begin{itemize}
\item[($\varPhi_{\lambda}$)]
There exist $s >0$, $R >0$ and  continuous, non-decreasing function
$\lambda(r) \in(0,1)$ on the interval $(0,R)$, $\lim\limits_{r \rightarrow 0} \lambda(r)=0$, $\lim\limits_{r \rightarrow 0} \dfrac{r}{\lambda(r)}=0$,   such that for any $B_{r}(x_{0}) \subset B_{R}(x_{0}) \subset \Omega$ and some
$A >0 $ there holds
$$
\varPhi^{+}_{B_{r}(x_{0})}\bigg( \frac{\lambda(r) v}{r^{1+\frac{n}{s}}}\bigg)\leqslant A\,\, \varPhi^{-}_{B_{r}(x_{0})}\bigg(\frac{\lambda(r) v}{r^{1+\frac{n}{s}}}\bigg),\quad r^{1+\frac{n}{s}} \leqslant \lambda(r) v\leqslant 1,
$$
here $\varPhi^{+}_{B_{r}(x_{0})}( v):=\sup\limits_{x\in B_{r}(x_{0})}\varPhi(x, v),\quad  \varPhi^{-}_{B_{r}(x_{0})}( v):=\inf\limits_{x\in B_{r}(x_{0})}\varPhi(x, v),\quad v > 0 $.
\end{itemize}
For the function $\lambda(r)$ we also need the following condition
\begin{itemize}
\item[($\lambda$)]
For any $0< r < \rho< R$ there holds
\begin{equation*}
\lambda(r) \geqslant \lambda(\rho) \bigg(\frac{r}{\rho}\bigg)^{b},
\end{equation*}
with some $b \geqslant 0$.
\end{itemize}
For the function $\lambda(r)=\left[\log\dfrac{1}{r}\right]^{-\frac{\beta}{q-p}}$, $\beta \geqslant 0$ this condition holds evidently, provided that $R$ is small enough.
\begin{remark}\label{rem1.1}
Consider the function $\varPhi(x, v):= v^{p}+a(x)v^{q}$, $a(x) \geqslant 0$, $\osc\limits_{B_{r}(x_{0})}a(x) \leqslant K r^{a} \big[\log\frac{1}{r}\big]^{\beta}$, $a\in (0,1]$, $\beta \geqslant 0$, $K>0$. Evidently condition ($\varPhi_{\lambda}$) holds with $\dfrac{n(q-p)}{a+p-q} \leqslant s \leqslant \infty$, $a\geqslant q-p$, $\lambda(r):= \big[\log\frac{1}{r}\big]^{-\frac{\beta}{q-p}}$ and $A=K^{q-p}$.

For the function $\varPhi(x, v):=v^{p(x)}$, $\osc\limits_{B_{r}(x_{0})}p(x)\leqslant\dfrac{L}{\log\frac{1}{r}}$, $L>0$ condition ($\varPhi_{\lambda}$) holds with $s>0$, $\lambda(r)\equiv 1$ and $A=\exp\big(L(1+\frac{n}{s})\big)$.
\end{remark}
\begin{remark}\label{rem1.2}
We note that conditions ($\varPhi_{\lambda}$) and ($A1-s_{*}$) with $s_{*}=\dfrac{ns}{n+s}$ from \cite{BenHarHasKar} essentially coincide in the case $\lambda(r)\equiv 1$.
\end{remark}

We refer to the parameters $n$, $p$, $q$, $s$,  $c$, $A$   as our structural  data, and we write $\gamma$ if it can be quantitatively determined a priory in terms of the above quantities. The generic constant $\gamma$ may change from line to line.

Our main result reads as follows.
\begin{theorem}\label{th1.1}
Let $u\in DG^{-}(\Omega)$, $u\geqslant 0$, let conditions ($\varPhi$), ($\varPhi_{\lambda}$), ($\lambda$) be fulfilled. Let $B_{8\rho}(x_{0})\subset B_{R}(x_{0})\subset \Omega$, let additionally
$u \in L^{s}_{loc}(\Omega)$ with some $s\geqslant q-p$ and $\left(\int\limits_{B_{2\rho}(x_{0})} u^{s} \right)^{\frac{1}{s}} \leqslant d$. Then there exists a positive constant $C$ depending only
on the known parameters and $d$, such that
\begin{equation}\label{eq1.2}
\left(\,\,\fint\limits_{B_{\rho}(x_{0})} (u+\rho )^{\theta} dx\right)^{\frac{1}{\theta}}\leqslant \frac{C}{\lambda(\rho)} \left(\inf\limits_{B_{\frac{\rho}{2}}(x_{0})} u +\rho \right),
\end{equation}
where $\fint\limits_{B_{\rho}(x_{0})}u^{\theta} dx:= |B_{\rho}(x_{0})|^{-1} \int\limits_{B_{\rho}(x_{0})}u^{\theta} dx$ and $\theta >0$ is some fixed number  depending only on the known data.
\end{theorem}

The conditions of Theorem~\ref{th1.1} are precise, we refer the reader to \cite{BenHarHasKar} for the examples.
In the case  $s=\infty$, Theorem \ref{th1.1}  was proved in \cite{BarColMin1, SavSkrYev}.

The main difficulty arising in the proof of our main result, Theorem \ref{th1.1}, is related to the so-called theorem on the expansion of positivity. Roughly speaking, having information on the measure of the ''positivity set'' of $u$ over the ball $B_{r}(\bar{x})$:
$$|\{x \in B_{r}(\bar{x}) : u(x) \geqslant N \}| \geqslant \alpha|B_{r}(\bar{x})|,$$
with some $r, N >0$ and  $\alpha \in (0, 1)$, we need to  translate it into the expansion of the set of positivity  to a ball $B_{2r}(\bar{x})$. Difficulties arise not only due to the presence of a factor $\lambda(r) $ in condition ($\varPhi_{\lambda}$), but also due to the presence of  the second term on the right-hand side of inequality \eqref{eq1.1}.  We do not use the classical covering argument of  Krylov and Safonov \cite{KrySaf}, DiBenedetto and  Trudinger \cite{DiBTru} as it was done in the ''bounded'' case, i.e. if $s= + \infty$ (see e.g. \cite{BarColMin1}), instead we use the  local clustering lemma due to DiBenedetto, Gianazza, and Vespri \cite{DiBGiaVes} and moreover, instead of $\sup\limits_{B_{2r}(\bar{x})}u$ we are forced to use averages of $u$ over the ball $B_{2r}(\bar{x})$.

The rest of the paper contains proof of the above theorem. In Section \ref{Sec2} we collect some auxiliary propositions and required integral estimates of functions belonging to the corresponding De Giorgi classes. Section \ref{Sec3} contains the proof of weak Harnack inequality, Theorem~\ref{th1.1}.
%%%%%%%%%%%%%%%%%%%%%%%%%%%%%%%%%%%%%%%%%%%%%%%%%%%%%%%%%%%%%%%%%%%%%%%%%%%%%%%%%%%%%%%%%%%%%%%%%%%%%%%%%%%%%%%
%%%%%%%%%%%%%%%%%%%%%%%%%%%%%%%%%%%%%%%%%%%%%%%%%%%%%%%%%%%%%%%%%%%%%%%%%%%%%%%%%%%%%%%%%%%%%%%%%%%%%%%%%%%%%%%
%%%%%%%%%%%%%%%%%%%%%%%%%%%%%%%%%%%%%%%%%%%%%%%%%%%%%%%%%%%%%%%%%%%%%%%%%%%%%%%%%%%%%%%%%%%%%%%%%%%%%%%%%%%%%%%%%%

\section{Auxiliary Material and Integral Estimates}\label{Sec2}

\subsection{ Local Clustering Lemma}\label{subsect2.1}

The following  lemma will be used in the sequel, it is  the local clustering lemma, see \cite{DiBGiaVes}.

\begin{lemma}\label{lem2.1}
Let $K_{r}(y)$ be a cube in $\mathbb{R}^{n}$ of edge $r$ centered at $y$ and let $u\in W^{1,1}(K_{r}(y))$ satisfy
\begin{equation}\label{eq2.1}
||(u-k)_{-}||_{W^{1,1}(K_{r}(y))} \leqslant \mathcal{K}\,k\,r^{n-1},\,\,\,\,\,\,and\,\,\,\,\,\, |\{K_{r}(y) : u\geqslant k \}|\geqslant \alpha |K_{r}(y)|,
\end{equation}
with some $\alpha\in (0,1)$, $k\in\mathbb{R}^{1}$ and $\mathcal{K} >0$. Then for any $\xi \in (0,1)$ and any $\nu\in (0,1)$ there exists $\bar{x} \in K_{r}(y)$ and $\delta=\delta(n) \in (0,1)$ such that
\begin{equation}\label{eq2.2}
|\{K_{\bar{r}}(\bar{x}):  u\geqslant \xi\,k \}| \geqslant (1-\nu) |K_{\bar{r}}(y)|,\,\,\, \bar{r}:=\delta \alpha^{2}\frac{(1-\xi)\nu}{\mathcal{K}}\,r.
\end{equation}
\end{lemma}

\subsection{Local Energy Estimates}\label{subsec2.2}

The following lemma is a consequence of inequalities \eqref{eq1.1}.
\begin{lemma}\label{lem2.2}
Let $u\in DG^{-}(\Omega$), $u\geqslant 0$,  $B_{r}(\bar{x})\subset B_{\rho}(x_{0}) \subset B_{8\rho}(x_{0})  \subset \Omega$, and let condition ($\varPhi$) holds, then for  any $k> 0$, any $\sigma \in(0,1)$ there holds
\begin{equation}\label{eq2.3}
\int\limits_{B_{r(1-\sigma)}(\bar{x})} | \nabla (u-\lambda(r) k)_{-}|^{p} \, dx \leqslant \gamma \sigma^{-q}
\frac{\varPhi^{+}_{B_{r}(\bar{x})}\big( \frac{\lambda(r) k}{r}\big)}{\varPhi^{-}_{B_{r}(\bar{x})}\big( \frac{\lambda(r) k}{r}\big)} \bigg(\frac{\lambda(r)k}{r}\bigg)^{p} |A^{-}_{\lambda(r)k,r}|.
\end{equation}

If additionally condition ($\varPhi_\lambda$) holds and
\begin{equation}\label{eq2.4}
r \leqslant \lambda(r)  k \leqslant \frac{1}{r^{\frac{n}{s}}},
\end{equation}
then
\begin{equation}\label{eq2.5}
\int\limits_{B_{r(1-\sigma)}(\bar{x})} | \nabla (u-\lambda(r) k)_{-}|^{p} \, dx \leqslant \gamma \sigma^{-q}
 \bigg(\frac{\lambda(r)k}{r}\bigg)^{p} |A^{-}_{\lambda(r)k,r}|.
\end{equation}
\end{lemma}
\begin{proof}
First, note the following Young's inequality
\begin{equation*}
\varPhi_{p}(x,a) b^{p} \leqslant \varPhi(x,a) + \varPhi(x,b),\quad a,b >0, \quad \varPhi_{p}(x,a):=a^{-p}\varPhi(x,a),
\end{equation*}
indeed,  if $b\leqslant a$, then $\varPhi_{p}(x,a) b^{p} \leqslant \varPhi(x,a)$ and if $b\geqslant a$, using the fact that by condition ($\varPhi$), function $\varPhi_{p}(x,a)$ is increasing, we obtain  $\varPhi_{p}(x,a) b^{p} \leqslant \varPhi(x,b)$.

Using this Young's inequality and inequalities \eqref{eq1.1} we get
\begin{multline*}
\int\limits_{B_{r(1-\sigma)}(\bar{x})} \varPhi^{-}_{B_{r}(\bar{x})}\bigg(\frac{\lambda(r) k}{r}\bigg)| \nabla (u-\lambda(r) k)_{-}|^{p} \, dx
\leqslant\\ \leqslant \bigg(\frac{\lambda(r) k}{r}\bigg)^{p}\int\limits_{B_{r(1-\sigma)}(\bar{x})}\varPhi_{p}\bigg(x,\frac{\lambda(r) k}{r}\bigg)|\nabla (u-\lambda(r)k)_{-}|^{p} dx \leqslant \\ \leqslant
\bigg(\frac{\lambda(r) k}{r}\bigg)^{p}\bigg\{\varPhi^{+}_{B_{r}(\bar{x})}\bigg(\frac{\lambda(r) k}{r}\bigg)|A^{-}_{\lambda(r)k,r}|+
\int\limits_{B_{r(1-\sigma)}(\bar{x})}\varPhi\big(x,|\nabla (u-\lambda(r)k)_{-}|\big) dx\bigg\} \leqslant \\ \leqslant
\gamma \sigma^{-q}\bigg(\frac{\lambda(r) k}{r}\bigg)^{p}\varPhi^{+}_{B_{r}(\bar{x})}\bigg(\frac{\lambda(r) k}{r}\bigg)|A^{-}_{\lambda(r)k,r}|,
\end{multline*}
which proves \eqref{eq2.3}. To prove \eqref{eq2.5} we note that by condition ($\varPhi_{\lambda}$)
\begin{equation}\label{eq2.6}
\frac{\varPhi^{+}_{B_{r}(\bar{x})}\big( \frac{\lambda(r) k}{r}\big)}{\varPhi^{-}_{B_{r}(\bar{x})}\big( \frac{\lambda(r) k}{r}\big)}
\leqslant \gamma,
\end{equation}
provided that $r^{1+\frac{n}{s}} \leqslant \lambda(r) r^{\frac{n}{s}}  k \leqslant 1$, which proves \eqref{eq2.5}. This completes the proof of the lemma.
\end{proof}

\subsection{Expansion of the Positivity }\label{subsec2.3}
The following lemma is a consequence of Lemma \ref{lem2.2} and Lemmas $6.2$, $6.3$ from \cite[Chap.\,2]{LadUra}.

\begin{lemma}\label{lem2.3}
Let $u\in DG^{-}(\Omega$), $u\geqslant 0$,  $B_{r}(\bar{x})\subset B_{\rho}(x_{0}) \subset B_{8\rho}(x_{0})  \subset \Omega$,
assume that the number $k>0$ satisfies the condition
\begin{equation}\label{eq2.7}
 \lambda(r)  k \leqslant \frac{1}{r^{\frac{n}{s}}},
\end{equation}
and assume also that with some $\alpha_{0} \in(0,1)$ there holds
\begin{equation}\label{eq2.8}
|\big\{B_{\frac{r}{2}}(\bar{x}): u \geqslant \lambda(r) k\big\}| \geqslant \alpha_{0} |B_{\frac{r}{2}}(\bar{x})| ,
\end{equation}
then there exists number $\eta_{0}\in(0,1)$, depending only on the data and $\alpha_{0}$, such that either
\begin{equation}\label{eq2.9}
\lambda(r) k \leqslant \frac{r}{\eta_{0} },
\end{equation}
or
\begin{equation}\label{eq2.10}
u(x) \geqslant \eta_{0}\,\lambda(r)\,k,\quad x \in B_{r}(\bar{x}).
\end{equation}

\end{lemma}
The proof of the lemma is almost standard. If \eqref{eq2.9} is violated, then by \eqref{eq2.7} inequalities  \eqref{eq2.4} hold. So, \eqref{eq2.5} define the standard De Giorgi classes $DG^{-}_{p}(\Omega)$ with the appropriate choice of number $k$. We refer the reader for the details to Lemmas $6.2$ and $6.3$ of \cite[Chap.\,2]{LadUra}.

\section{Weak  Harnack Inequality, Proof of Theorem \ref{th1.1}}\label{Sec3}

First, we prove the following lemma.
\begin{lemma}\label{lem3.1}
Let $B_{\rho}(x_{0}) \subset B_{8\rho}(x_{0}) \subset \Omega$, $\left(\int\limits_{B_{2\rho}(x_{0})} u^{s} \right)^{\frac{1}{s}} \leqslant d$ and let the following inequality holds
\begin{equation}\label{eq3.1}
\left|\left\{\, B_{\frac{\rho}{2}}(x_{0}): u \geqslant N \,\right\}\right| \geqslant \alpha \left|B_{\frac{\rho}{2}}(x_{0})\right|,
\end{equation}
for some $N >0$ and some  $\alpha \in(0,1)$. Then, under the conditions of Lemma~\ref{lem2.2}, there exist $C_{2}$, $\tau >0$ depending only on the data and $d$ such that either
\begin{equation}\label{eq3.2}
\alpha^{\tau}\leqslant  \frac{C_{2} \rho}{N \lambda(\rho)},
\end{equation}
or
\begin{equation}\label{eq3.3}
\alpha^{\tau}\leqslant \frac{C_{2}}{N \lambda(\rho)} \inf\limits_{B_{\frac{\rho}{2}}(x_{0})} u.
\end{equation}
\end{lemma}
\begin{proof}
Further we will assume that inequality \eqref{eq3.2} is violated, i.e. with some $\tau >0$
\begin{equation}\label{eq3.4}
\alpha^{\tau} N \lambda(\rho) \geqslant  C_{2} \rho.
\end{equation}
Let $\varepsilon\in(0,1)$ be some number to be chosen later. Applying inequality \eqref{eq2.3}  for \\ $(u- \lambda(\rho)\varepsilon N)_{-}$ over the pair of balls $B_{\frac{\rho}{2}}(x_{0})$ and $B_{\rho}(x_{0})$, we obtain
\begin{equation}\label{eq3.5}
\fint\limits_{B_{\frac{\rho}{2}}(x_{0})} |\nabla (u- \lambda(\rho)\varepsilon N)_{-}|^{p}\,dx \leqslant \gamma \,\frac{\varPhi^{+}_{B_{\rho}(x_{0})}\big( \frac{\lambda(\rho)\varepsilon N}{\rho}\big)}{\varPhi^{-}_{B_{\rho}(x_{0})}\big( \frac{\lambda(\bar{\rho})\varepsilon N }{\rho}\big)}\bigg(\lambda(\rho)\frac{\varepsilon N}{\rho}\bigg)^{p}  .
\end{equation}
Now we need to estimate the right-hand side of the last inequality, for this we note that inequality \eqref{eq3.1} yields
\begin{equation}\label{eq3.6}
\left(\fint\limits_{B_{2\rho}(x_{0})} u^{s}\,dx\right)^{\frac{1}{s}}\geqslant  4^{-\frac{n}{s}} \alpha^{\frac{1}{s}} N,
\end{equation}
and moreover
\begin{equation*}
|\{\, B_{\frac\rho2}(x_{0}) : u \geqslant \lambda(\rho)(1+d)^{-1} 4^{-\frac{n}{s}}\alpha^{\frac{1}{s}} N \,\}|\geqslant |\{\, B_{\frac\rho2}(x_{0}) : u \geqslant  N \,\}|\geqslant \alpha |B_{\frac\rho2}(x_{0})|.
\end{equation*}
Choosing $\varepsilon =  (1+d)^{-1} 4^{-\frac{n}{s}} \alpha^{\frac{1}{s}}$, by \eqref{eq3.4}, \eqref{eq3.6}, we obtain
\begin{equation}\label{eq3.7}
\rho \leqslant \lambda(\rho) \varepsilon N \leqslant (1+d)^{-1} \left(\fint\limits_{B_{2\rho}(x_{0})} u^{s}\,dx\right)^{\frac{1}{s}}
\leqslant \frac{d}{(d+1)\rho^{\frac{n}{s}}}\leqslant  \frac{1}{\rho^{\frac{n}{s}}},
\end{equation}
provided that $\tau \geqslant \dfrac{1}{s}$ and $C_{2} \geqslant (1+d) 4^{\frac{n}{s}}$.  Therefore Lemma \ref{lem2.2} and inequalities \eqref{eq3.5}, \eqref{eq3.7} yield
\begin{equation}\label{eq3.8}
\fint\limits_{B_{\frac{\rho}{2}}(x_{0})} |\nabla (u- \lambda(\rho)\varepsilon N)_{-}|^{p}\,dx \leqslant \gamma \,\bigg(\lambda(\rho)\frac{\varepsilon N}{\rho}\bigg)^{p}.
\end{equation}

The local clustering Lemma \ref{lem2.1} with $k=\lambda(\rho)\varepsilon  N$, $\nu=\frac{1}{4}$, $\xi=\frac{1}{4}$, $\mathcal{K}=\gamma $, $r=\frac{\rho}{2}$  implies the existence of a point $\bar{x}\in B_{\frac{\rho}{2}}(x_{0})$ and $\delta \in(0,1)$ depending only on the data, such that
\begin{equation*}
\left|\left\{B_{\bar{r}}(\bar{x}): u \geqslant \frac{\lambda(\bar{r})}{4} \varepsilon N \right\}\right|\geqslant \left|\left\{B_{\bar{r}}(\bar{x}): u \geqslant \frac{\lambda(\rho)}{4} \varepsilon N \right\}\right| \geqslant \frac{3}{4} |B_{\bar{r}}(\bar{x})|,\quad \bar{r}= \delta_{0} \alpha^{2} \rho, \quad \delta_{0}:= \frac{3\delta}{16\mathcal{K}}.
\end{equation*}
Since $\lambda(r)$ is non-decreasing, inequality \eqref{eq3.7} implies
\begin{equation}\label{eq3.9}
\frac{\lambda(\bar{r})}{4}\varepsilon N \leqslant \frac{\lambda(\rho)}{4}\varepsilon N \leqslant \frac{1}{\rho^{\frac{n}{s}}} \leqslant
\frac{1}{\bar{r}^{\frac{n}{s}}},
\end{equation}
and moreover, by condition ($\lambda$) and by \eqref{eq3.4}
\begin{equation*}
\frac{\lambda(\bar{r})}{4}\varepsilon N \geqslant \frac{\lambda(\rho)}{4}\bigg(\frac{\bar{r}}{\rho}\bigg)^{b}\varepsilon N
=\lambda(\rho)\frac{\delta_{0}^{b}}{4^{1+\frac{n}{s}}}(1+d)^{-1} \alpha^{\frac{1}{s}+ 2b} N \geqslant \rho\geqslant \bar{r},
\end{equation*}
provided that $\tau \geqslant 2b +\dfrac{1}{s}$ and $C_{2}\geqslant (1+d) \dfrac{4^{1+\frac{n}{s}}}{\delta_{0}^{b}}$.

So, Lemma \ref{lem2.3} is applicable with $\alpha_{0}=\dfrac{3}{4}$ and $k=\dfrac{\varepsilon}{4} N$, so we obtain with some $\eta_{0} \in(0,1)$ depending only on the data
$$
u(x)\geqslant \frac{\eta_{0}\lambda(\bar{r})\varepsilon N}{4}, \quad x \in B_{2\bar{r}}(\bar{x}),
$$
provided that $C_{2}\geqslant (1+d)\dfrac{4^{1+\frac{n}{s}}}{\eta_{0}\delta_{0}^{b}}$.

Repeating this procedure $j$ times we obtain
\begin{equation}\label{eq3.10}
 u(x)\geqslant \frac{\eta_{0}^{j} \,\lambda(\bar{r})\varepsilon N}{4}, \quad x \in B_{2^{j}\bar{r}}(\bar{x}),
\end{equation}
provided that
\begin{equation}\label{eq3.11}
2^{j} \bar{r}\leqslant \frac{\eta_{0}^{j} \,\lambda(\bar{r})\varepsilon N}{4}\leqslant \dfrac{1}{(2^{j}\bar{r})^{\frac{n}{s}}}.
\end{equation}
 Choose $j$ by the condition $2^{j}\bar{r}=\rho$, that is $2^{j} \delta_{0} \alpha^{2}=1$, the second inequality in \eqref{eq3.11} holds by \eqref{eq3.9}. By \eqref{eq3.4} and  condition ($\lambda$), we have
\begin{equation*}
\eta_{0}^{j} \lambda(\bar{r}) \varepsilon N \geqslant \frac{\delta_{0}^{\log\frac{1}{\eta_{0}}+b}}{(1+d) 4^{\frac{n}{s}}}\,\alpha^{2\log\frac{1}{\eta_{0}}+ 2b+\frac{1}{s}}\,N\,\lambda(\rho) \geqslant \rho= 2^{j} \bar{r},
\end{equation*}
provided that  $\tau=2 \log\dfrac{1}{\eta_{0}}  + 2b +\dfrac{1}{s}$  and $C_{2}=C_{2}(d,\eta_{0},\delta_{0})>0$ is large enough.  Therefore,
inequality \eqref{eq3.10} yields
\begin{equation*}
u(x)\geqslant \frac{\lambda(\rho)}{C_{2}} \alpha^{\tau} N, \quad  x \in B_{\frac{\rho}{2}}(x_{0}),
\end{equation*}
which completes the proof of the lemma.
\end{proof}
To complete the proof of the weak Harnack inequality,  we set  $\bar{m}(\rho)=\dfrac{1}{\lambda(\rho)}\left( \inf\limits_{B_{\frac{\rho}{2}}(x_{0})}u(x)+\rho \right),$ then Lemma \ref{lem3.1}  with  $\theta\in\big(0, \frac{1}{2\tau}\big]$  yields
\begin{multline*}
\fint\limits_{B_{\rho}(x_{0})}u^{\theta}\,dx=\frac{\theta}{|B_{\rho}(x_{0})|}\,\int\limits_{0}^{\infty}|\{B_{\rho}(x_{0}): u(x) >N\}|\,N^{\theta-1}\,dN \leqslant \\ \leqslant  [\bar{m}(\rho)]^{\theta}+ \gamma [\bar{m}(\rho)]^{\frac{1}{\tau}}\int\limits_{\bar{m}(\rho)}^{\infty} N^{\theta-\frac{1}{\tau}-1} \,dN \leqslant \gamma [\bar{m}(\rho)]^{\theta},
\end{multline*}
which proves Theorem \ref{th1.1}.

\vskip3.5mm
{\bf Acknowledgements.} The authors are partially supported due to the project "Mathematical \,\,modeling\,\, of\, complex\,\, dynamical\,\, systems\,\, and\,\, processes\,\, caused\,\, by\,\, the\,\, state\,\, security" (Reg. No. 0123U100853), by the grant of Ministry of Education and Science of Ukraine (Reg. No. 0121U109525) and by the Grant EFDS-FL2-08 of the found The European Federation of Academies of Sciences and Humanities (ALLEA).

%\bigskip
\newpage
CONTACT INFORMATION

\medskip
\textbf{Mariia O.~Savchenko}\\Institute of Applied Mathematics and Mechanics,
National Academy of Sciences of Ukraine, \\ \indent Heneral Batiuk Str. 19, 84116 Sloviansk, Ukraine\\
shan\underline{ }maria@ukr.net

\medskip
\textbf{Igor I.~Skrypnik}\\Institute of Applied Mathematics and Mechanics,
National Academy of Sciences of Ukraine, \\ \indent Heneral Batiuk Str. 19, 84116 Sloviansk, Ukraine\\
Vasyl' Stus Donetsk National University,
\\ \indent 600-richcha Str. 21, 21021 Vinnytsia, Ukraine\\ihor.skrypnik@gmail.com

\medskip
\textbf{Yevgeniia A. Yevgenieva}
\\ Max Planck Institute for Dynamics of Complex Technical Systems, \\ \indent Sandtorstrasse 1, 39106 Magdeburg, Germany
\\Institute of Applied Mathematics and Mechanics,
National Academy of Sciences of Ukraine, \\ \indent Heneral Batiuk Str. 19, 84116 Sloviansk, Ukraine\\yevgeniia.yevgenieva@gmail.com

\end{document}